\documentclass[11pt]{article}
\usepackage{geometry}                
\usepackage{amsmath, amsthm, amssymb}
\usepackage{graphicx, color}
\usepackage{pdfsync}
\newtheorem{lemma}{Lemma}

\newtheorem{proposition}[lemma]{Proposition}
\newtheorem{remark}[lemma]{Remark}

\newtheorem{theorem}[lemma]{Theorem}

\usepackage{graphicx}
\usepackage{amssymb}
\usepackage{epstopdf}
\newcommand{\e}{\varepsilon}
  
\newcommand{\disp}{\displaystyle}

\def\ZZ{\mathbb Z}
\def\NN{\mathbb Z}
\def\RR{\mathbb R}

\def\dist{\hbox{\rm dist}}

\title{Motion of discrete interfaces through mushy layers}
\author{Andrea Braides \\Dipartimento di Matematica, Universit\`a di Roma Tor Vergata
\\ via della ricerca scientifica 1, 00133 Roma, Italy\\ \\ Margherita Solci\\
DADU, Universit\`a di Sassari\\
 piazza Duomo 6, 07041 Alghero (SS), Italy}
\date{}                                           

\begin{document}
\maketitle

\abstract \noindent We study the geometric motion of sets in the plane derived from 
the homogenization of discrete ferromagnetic energies with weak inclusions.
We show that the discrete sets are composed by a `bulky' part and an external
`mushy region' composed only of weak inclusions. The relevant motion is 
that of the bulky part, which asymptotically obeys to a motion by crystalline mean 
curvature with a forcing term, due to the energetic contribution
of the mushy layers, and pinning effects, due to discreteness. From an analytical standpoint
it is interesting to note that the presence of the mushy layers imply only a weak
and not strong convergence of the discrete motions, so that the convergence
of the energies does not commute with the evolution. From a mechanical standpoint
it is interesting to note the geometrical similarity of some phenomena in
the cooling of binary melts.

\section{Introduction}
A definition of motion by curvature has been introduced by Almgren, Taylor and Wang \cite{ATW}
using a time-discrete approach as follows. Given a (smooth) set $A_0\subset\RR^d$ as initial datum
and a time scale $\tau$, one defines iteratively $A^{\tau}_k$ as a minimizer of
\begin{equation}
A\mapsto {\rm Per}(A)+{1\over \tau} D(A,A^{\tau}_{k-1}),
\end{equation}
where $A^\tau_0=A_0$ , ${\rm Per}(A)$ denotes the Euclidean perimeter of the set 
$A$ and $D(A,A')$ is a dissipation 
term that can be interpreted as the $L^2$-norm of the distance between $\partial A$ and 
$\partial A'$. The time-continuous piecewise-constant interpolations 
$A^\tau(t)= A^\tau_{\lfloor t/\tau\rfloor}$ are then shown to converge to a time-continuous
parameterized sets $A(t)$, whose boundaries move by their mean curvature. Almgren and 
Taylor \cite{AT} have shown that the same scheme with a crystalline perimeter gives
motion by crystalline curvature in dimension two. 

The same scheme has been adapted to define a continuum motion for ferromagnetic spin energies
on a square lattice whose static discrete-to-continuum approximation is a crystalline perimeter by Braides,
Gelli and Novaga \cite{BGN}. In this process a scaling factor $\e>0$ and the corresponding 
perimeter  ${\rm Per}_\e(A)$ for discrete sets $A$ in $\e\ZZ^2$ are introduced, together with 
the corresponding discrete dissipations $D_\e$. These can be seen simply as the restriction of their continuum
counterparts union $A+[0,\e]^2$ of $\e$-cubes. In this way discrete sets $A^{\e,\tau}_{k}\subset \e\ZZ^2$ 
are defined by iterated minimization of 
\begin{equation}\label{ATWe}
A\mapsto {\rm Per}_\e(A)+{1\over \tau} D_\e(A,A^{\e,\tau}_{k-1}),
\end{equation}
where $A^{\e,\tau}_0$ are discrete interpolations of a continuum datum $A_0$
together with the corresponding piecewise-constant-in-time interpolations 
$A^{\e,\tau}(t)=A^{\e,\tau}_{\lfloor t/\tau\rfloor}$. 
The limit $A$ of $A^{\e,\tau}$ may depend on the mutual behaviour 
of $\e$ and $\tau$. 

This procedure can be framed in the theory of {\em minimizing movements} by De Giorgi (see \cite{AGS}), which 
generalizes the approach of \cite{ATW}. In \cite{B-LN} a notion of {\em minimizing movement along a
sequence of functionals} has been given, that can be specialized for possibly inhomogeneous perimeter-type 
energies $F_\e$ on $\e\ZZ^2$: substituting $ {\rm Per}_\e$ with $F_\e$ in the scheme above we can similarly 
define $A^{\e,\tau}$ and obtain a limit motion $A(t)$ passing to the limit both in $\e$ and $\tau$.
In particular the following holds, upon the hypothesis of equi-coerciveness
of $F_\e$ and their $\Gamma$-convergence to some $F$:

(i) ({\em pinning}) if $\tau<\!< \e$ then $A(t)=A_0$ for all $t\ge 0$;

(ii) ({\em commutation})  if $ \e<\!<\tau$ then $A(t)$ is the minimizing movement of $F$ 
with initial datum $A_0$ (hence, in the case of $F_\e= {\rm Per}_\e$ the sets $A$ move by crystalline curvature);

(iii) ({\em critical scale}) if $ \e\sim\tau$ then the motion actually depends on $\e/\tau$ 
and is different both from (1) or (2).

In \cite{BGN} this last case is explicitly described through an {\em effective motion}: the limit
$A(t)$ depends on the ratio $\e/\tau$, and this motion may 
depend on fine details of the energies $F_\e$ and not only through their 
$\Gamma$-limit (see \cite{BScilla,Scilla}). The mechanism of evolution highlighted by 
the definition of $A^{\e,\tau}_k$ is through local minimization of $F_\e$ with a 
dissipation contribution that forces minimization on a small neighbourhood of the 
datum $A^{\e,\tau}_{k-1}$. If $\tau$ is much smaller than $\e$ then by the discreteness
of the parameters this neighbourhood contains the only $A^{\e,\tau}_{k-1}$, and the
motion is pinned. Conversely, if $\e$ is much smaller than $\tau$ we can first pass
in the limit as $\e\to0$ in (\ref{ATWe}) and use the well-known property of convergence
of minimum problems for $\Gamma$-convergence. In the critical case the motion
optimizes the location of the interface combining energy and dissipation effects.

\medskip
In this paper we consider $F_\e$ a sequence of {\em double-porosity} type energies,
which mix `strong'  ferromagnetic interactions with `weak' inclusions 
and still $\Gamma$-converge to a crystalline perimeter (see \cite{BCPS}). 
In this case the mechanism governing the time-discrete motion is of a different type
from \cite{BGN,BScilla}: since the perimeter energy due to weak inclusions 
in $A^{\e,\tau}_{k-1}$ may be small with respect to the dissipation necessary to 
remove them from $A^{\e,\tau}_{k}$, the latter is composed of a `bulky' part,
plus an external {\em mushy layer} composed of weak inclusions. This 
suggestive terminology is borrowed from theories in Fluid Mechanics where
similar geometries appear in binary melts at solidification \cite{HW,W1,W2}.
These mushy 
layer may then disappear at the next step. As a result the final motion is not a simple
motion by crystalline curvature, but it also contains a forcing term as a result
of the effect of the mushy layers. This is true also for $\e<\!<\tau$ for which we have
the law of motion
$$
v= (a\kappa-b)^+
$$
relating the velocity and the crystalline curvature. Note in particular that the conclusion 
(ii) above is violated. This is explained by a loss of coerciveness of the energies
$F_\e$ as $\e\to0$.

\section{Discrete setting and statement of the problem}\label{insiemi}
We are interested in describing a geometric continuum motion 
derived as the limit of time-discrete motions defined for discrete sets
of $\e\ZZ^2$ as $\e\to 0$ in the spirit of {\em minimizing movements
along a sequence of energies} \cite{B-LN}. In the specific two-dimensional
case we are dealing with, the relevant information about the limit motion
is obtained by considering initial data which are coordinate rectangles.
The motion for more general sets can be derived from that case and
obeys the same motion by crystalline curvature \cite{AT,BGN}.

We will examine the time-discrete motions at fixed $\e$ and let eventually $\e\to 0$.
For the sake of simplicity of notation, we
will state our problems in terms discrete subsets $I$ in $\ZZ^2$ without scaling them by $\e$,
but keep in mind that we are interested in the corresponding scaled sets $\e I$.

With fixed $\alpha, \beta>0$, for $\e>0$ we define on $I\subset \ZZ^2$ the energy 
\begin{equation}\label{def-energia} 
F_\e(I)=\e\beta\left\{(i,j)\in \hbox {NN}_s: \ i\in I, j\not\in I\right\} 
+\e^2\alpha\left\{(i,j)\in \hbox {NN}_w: \ i\in I, j\not\in I\right\},
\end{equation}
where 
\begin{eqnarray}
&&\hbox {NN}_s=\left\{(i,j)\in \ZZ^2\times\ZZ^2: \ \|i-j\|_\infty=1 \hbox{ and } i_1=j_1 \hbox{ odd or } i_2=j_2 \hbox{ odd} \right\}\nonumber \\
&&\hbox {NN}_w=\left\{(i,j)\in \ZZ^2\times\ZZ^2: \ \|i-j\|_\infty=1 \hbox{ and } i_1=j_1 \hbox{ even or } i_2=j_2 \hbox{ even} \right\}.\nonumber 
\end{eqnarray}
Besides these energies we introduce discrete dissipations, which take into account the
$L^2$ distance of the boundaries of the discrete sets, given by
\begin{equation}\label{def-dist} 
D_\e(I, I^\prime)=\e^3 \Bigl(\sum_{i\in I^\prime\setminus I} \dist_\infty(i,c(I^\prime)) 
+ \sum_{i\in I\setminus I^\prime} \dist_\infty(i,I^\prime) \Bigr), 
\end{equation}
where $I, I^\prime\subset \ZZ^2$ and $c(I^\prime)=\ZZ^2 \setminus I^\prime$.  

Let $\e>0$ and $\tau>0$ be fixed, together with 
a discrete coordinate rectangle $I_0$ with the four vertices in $(2\ZZ^2)$ 
(we omit the possible dependence on $\e$). 
We construct a sequence $\{I_n\}$,  where $I_n$ minimizes 
\begin{eqnarray}
E_\e(I, I_{n-1})=F_\e(I)+\frac{1}{\tau}D_\e(I,I_{n-1}).\label{n-1}
\end{eqnarray} 
The following lemma describes the relevant properties of the minimizers of $(\ref{n-1})$. 
\begin{lemma}\label{prop}
For each $n$, if $I_n$ is a minimizer of {\rm(\ref{n-1})} then 
\begin{enumerate}
\item[{\rm a}.] $I_n\subseteq I_{n-1}$; 
\item[{\rm b}.] there exist a discrete coordinate rectangle $R_n$ with the four vertices in $(2\ZZ)^2$ and a set $W_n\subset (2\ZZ)^2$ 
such that $I_n=R_n\cup W_n$.
\end{enumerate}
\end{lemma}

This lemma (whose proof is postponed to the next section) shows that the discrete evolution of the sets $I^{\e,\tau}_n=\e I_n$ that we are interested in is described by a {\em bulky part} governed by the coordinate rectangles $R^{\e,\tau}_n=\e R_n$, and a {\em mushy layer} composed of {\em weak islands} given by $W^{\e,\tau}_n=\e W_n$.

A relevant parameter in the description of the limit evolution is the {\em ratio of time and space scales}
\begin{equation}\label{gammadef}
\gamma=\lim_{\e\to 0}{\tau\over\e},
\end{equation}
where it is understood that $\tau=\tau(\e)$.

We will show that the asymptotic description of the sets $W^{\e,\tau}_n$ is not necessary 
to characterize the limit. Indeed

$\bullet$ if $4\alpha\gamma<1$ then $W_n=((2\ZZ)^2\cap I_0)\setminus R_n$; i.e., the mushy layer always contains all the weak sites in the initial datum $I_0$;

$\bullet$ if $4\alpha\gamma>1$ then $W_n$ is contained in $R_{n-1}$; i.e., the mushy layer at time step $n-1$ disappear at the next step.

The case $4\alpha\gamma=1$ is exceptional, as in this case it may be equivalent in terms of the balance between energy and dissipation to maintain weak islands or `dissipate' them. 

In any case, the relevant asymptotic description is given by the following result, whose proof is the content of the rest of the paper. In the description we do not treat in detail some non-uniqueness cases highlighted by the discontinuous right-hand side of the ODE in (\ref{molo}), which anyhow are completely analogous to those dealt with in \cite{BGN}.

\begin{theorem} Let $\e$ and $\tau$ be fixed and let $R_0$ be a given rectangle with the length  of the horizontal side $L^1_0$ and  length  of the vertical side $L^2_0$. Let $I_0$ be the 
greatest discrete coordinate rectangle with the four vertices in $(2\ZZ^2)$ contained in ${1\over\e}R_0$
Let $R_n$ be the sequence of rectangles of $\ZZ^2$ constructed by successive minimization as in Lemma {\rm\ref{prop}}. We define $L^1_n$ and $L^2_n$ as the lengths of the horizontal and vertical sides of $R_n$, respectively, and $L^1_{\e,\tau}(t)= \e L^1_{\lfloor t/\tau\rfloor}$ and  $L^2_{\e,\tau}(t)= \e L^2_{\lfloor t/\tau\rfloor}$. Let $\e$ and $\tau$ tend to $0$ and {\rm(\ref{gammadef})} be satisfied; then $L^1_{\e,\tau}(t)$ and  $L^2_{\e,\tau}(t)$ tend to $L_1(t)$ and  $L_2(t)$, respectively, satisfying $L_1(0)= L^1_0$ and
\begin{equation}\label{molo}
L'_1(t)=-\frac{4}{\gamma}
\Big\lfloor\max\Bigl\{\frac{2\beta\gamma}{3L_2(t)}-\frac{2\alpha\gamma}{3}+\frac{1}{6},
\frac{\beta\gamma}{2L_2(t)}+\frac{1}{4}\Bigr\}\Big\rfloor
\end{equation}
for almost every $t$.
In particular, we have {\em pinning} (no variation of $L_1$) if
\begin{eqnarray}\label{molopin}
L_2(t)>{4\beta\gamma\over {4\alpha\gamma}+5} && \hbox{ if $4\alpha\gamma<1$, or}\\
L_2(t)>{2\beta\gamma\over 3} && \hbox{ if $4\alpha\gamma>1$,}
\end{eqnarray}
respectively.
The analog description holds for $L_2$.
\end{theorem} 

In terms of the crystalline curvature, which for a coordinate edge is given by $\kappa={2\over L}$
($L$ being its length), equation (\ref{molo}) reads as
\begin{equation}\label{molo}
v=\frac{2}{\gamma}
\Big\lfloor\max\Bigl\{\frac{\beta\gamma\kappa}{3}-\frac{2\alpha\gamma}{3}+\frac{1}{6},
\frac{\beta\gamma\kappa}{4}+\frac{1}{4}\Bigr\}
\Big\rfloor,
\end{equation}
where $v$ is the velocity of the edge. This equation can be extended to coordinate polyrectangles and then to more general sets by approximation \cite{BGN}.

\begin{remark}\rm
Note that for $4\alpha\gamma<1$ equation (\ref{molo}) simplifies to 
\begin{equation}\label{moloch}
L'_1(t)=-\frac{4}{\gamma}
\Big\lfloor\frac{2\beta\gamma}{3L_2(t)}-\frac{2\alpha\gamma}{3}+\frac{1}{6}\Big\rfloor.
\end{equation}
\end{remark}

\begin{remark}[extreme cases]\rm
We can consider the cases $\tau<\!<\e$ and $\e<\!<\tau$ by letting $\gamma\to 0$ and $\gamma\to +\infty$, respectively. In the first case we have {\em pinning} for all initial data, and the motion is trivial. 
For $\gamma\to+\infty$  the motion of $L_1$ (and similarly that of $L_2$) is described by
\begin{equation}\label{ligain}
L'_1(t)=-\max\Bigl\{\frac{8}{3}\Bigl({\beta\over L_2(t)}-\alpha \Bigr), {2\beta\over L_2(t)}\Bigr\}.
\end{equation}
\end{remark}

\begin{remark}[a non-commutability phenomenon]\rm
In \cite{BCPS} it is shown that the discrete energies $E_\e(I)= F_\e({1\over\e}I)$ defined on subsets of $\e\ZZ^2$ $\Gamma$-converge to the crystalline energy
$$
F(A)=\int_{\partial A} {\beta\over 2}\|\nu\|_1d{\cal H}^1,
$$
whose minimizing movements give motion by crystalline mean curvature  $v={\beta\over 2}\kappa$ (see Almgren and Taylor \cite{AT}), which corresponds in the case of a rectangle to side lengths satisfying
$$
L'_1(t)=-{2\beta\over L_2(t)}.
$$
A general result in \cite{B-LN} shows that there exist a sufficiently slow time scale $\tau$ such that
the minimizing movement of a sequence of energies $E_\e$ along $\tau$ gives the minimizing movement
of the $\Gamma$-limit. This is seemingly in contrast with the result in the theorem above since for $\gamma\to +\infty$ (which corresponds to slow time scales) we have equation (\ref{ligain}).
This discrepancy is explained by the lack of equicoerciveness of the energies. Indeed
the result in \cite{B-LN} only holds if the sequence is strongly equicoercive, and 
the appearance of the mushy region exactly corresponds to a weak (and not strong) 
convergence of the evolutions.
\end{remark}

\section{Description of the structure of minimizers}
This section is devoted to the proof of Lemma \ref{prop}.
To show the result, it is useful to give a notion of {\em connectedness} for a discrete set $I\subset \ZZ^2$. 
We say that a discrete set $I\subset\ZZ^2$ is connected if 
the set $a(I)=\bigcup_{i\in I}(i+[-\frac{1}{2},\frac{1}{2}]^2)$ is connected. 
Given $I, I^\prime$ discrete sets with $I\subseteq I^\prime$, $I$ is a connected component 
of $I^\prime$ if $a(I^\prime)$ is a connected component of $a(I)$. 

\begin{proof}[Proof of Lemma {\rm\ref{prop}}.] 
Note that setting $R_0=I_0$ and $W_0=\emptyset$ we have $I_0=R_0\cup W_0$.  
Suppose that such property holds for $n-1$.   

If $I_n$ is a minimizer of $(\ref{n-1})$ and $I_{n}\subset (2\ZZ)^2$, 
we can choose $R_n=\emptyset$ and $W_n=I_{n}$. 
If there exists $i\in I_{n}\setminus I_{n-1}$, 
then $E_\e(I_{n}\setminus i, I_{n-1})=E_\e(I_{n}, I_{n-1})-4\e^2\alpha-\e^3\tau^{-1}<E_\e(I_{n}, I_{n-1})$.
Since $I_n$ is a minimizer, this implies $I_{n}\subseteq I_{n-1}$ and the thesis follows. 

Now, we consider the case $I_n\setminus(2\ZZ)^2\neq\emptyset$.
Let $C$ be a connected component of $I_{n-1}$ 
such that $C\setminus (2\ZZ)^2\neq\emptyset$. 
We show by contradiction that $C\subseteq R_{n-1}$. 
Since in a connected set $I\subset\ZZ^2$ 
we have $\# (I\cap(2\ZZ)^2)\leq\# (I\setminus(2\ZZ)^2)+1$, 
setting $\tilde I_n=I_n\setminus (C\setminus R_{n-1})$ we get  
$$E_\e(\tilde I_n,I_{n-1})-E_\e(I_n,I_{n-1})<-\frac{\e^3}{\tau}-2\beta\e $$
which is negative for $\e$ small enough. 
Hence, $C\subseteq R_{n-1}$. 
Denoting by $R(C)$ the minimal discrete coordinate rectangle 
including $C$, since $F_\e(R(C))\leq F_\e(C)$, necessarily 
$C=R(C)$ and $C=([a,\overline a]\times[b,\overline b])\cap\ZZ^2$ 
for some $a,\overline a,b,\overline b\in\ZZ$. 
We show that the four vertices of $C$ belong to $(2\ZZ)^2$. 
Reasoning by contradiction, it is not restrictive to assume 
$C=([a,\overline a]\times [b,\overline b])\cap \ZZ^2$ with 
$a\not\in 2\ZZ$. Note that, setting $N=\#([b,\overline b]\cap\ZZ)$, 
we have $\sum_{i=b}^{\overline b}d((a,i),c(I_{n-1})) \geq 2N_\e-2.$
Since $\sum_{i=b}^{\overline b}d((a,i),c(I_{n-1}))\geq 2\beta\tau\e^{-2},$ 
setting $\tilde I_n=I_n\cup (\{a-1\}\times([b,\overline b]\cap\ZZ))$ 
it follows that 
\begin{eqnarray*}
E_\e(\tilde I_n, I_{n-1})-E_\e(I_n, I_{n-1})&\leq& 2\alpha\e^2-\frac{\e^3}{\tau}
\sum_{i=b}^{\overline b}d((a-1,i),c(I_{n-1}))\\
&\leq & 2\alpha\e^2-\frac{\e^3}{2\tau} 
\sum_{i=b}^{\overline b}d((a,i),c(I_{n-1})) +\frac{\e^3}{\tau} \\
&\leq & 2\alpha\e^2-\beta\e +\frac{\e^3}{\tau}
\end{eqnarray*} 
which is strictly negative for $\e$ small enough. 

Now, we show that the connected component 
$C=([a,\overline a]\times[b,\overline b])\cap \ZZ^2$ is the unique connected component 
intersecting $\ZZ^2\setminus(2\ZZ)^2$. 
We prove this by showing that the center $(n_1,n_2)$ of $R_{n-1}$ belongs to $C$. 
Since $C\setminus (2\ZZ)^2\neq\emptyset$, it is not restrictive to assume $b<\overline b$. 
If $\overline a<n_1$, we consider 
the set $\tilde I_{n}=I_n\cup (\{\overline a+1\} \times ([b,\overline b])\cap \ZZ)$. 
Since 
$$\e^3\tau^{-1}\sum_{i=b}^{\overline b} d((\overline a,i), c(I_n))\geq 
2 \e\beta + 6\e^2\alpha,$$
we get  
\begin{eqnarray*}
E_\e(\tilde I_{n}, I_{n-1})-E_\e(I_{n}, I_{n-1})&\leq& -\e^3\tau^{-1}\sum_{i=b}^{\overline b} d((\overline a,i), c(I_n))
+2\e\beta \\
&\leq& - 12\e^2\alpha < 0.
\end{eqnarray*}
The same argument holds for $a>n_1$. Hence, 
$a\leq n_1\leq \overline a$. Assume by contradiction that $\overline b<n_2$, and define 
$\tilde I_n=I_n\cup (([a,\overline a]\cap\ZZ)\times([\overline b+1, n_2]\cap\ZZ))$ 
(if $n_2\not\in 2\ZZ$, we substitute $n_2$ by $n_2+1$). 
Since 
$$\frac{\e^3}{\tau}\sum_{i=a}^{\overline a} \big(d((i, \overline b), c(I_n))+d((i, \overline b-1), c(I_n))\big)
 \geq 
2 \e\beta + 2\e^2\alpha,$$
we get, recalling that $d((n_1, n_2), c(I_n))<d((n_1, \overline b), c(I_n))$,   
\begin{eqnarray*}
E_\e(\tilde I_{n}, I_{n-1})-E_\e(I_{n}, I_{n-1})&<& 
-K\frac{\e^3}{\tau}\sum_{i=a}^{\overline a} \big(d((i, \overline b), c(I_n))+d((i, \overline b-1), c(I_n))\big)\\
&&+2K\e\beta +2K\e^2\alpha\\
&\leq& 0 
\end{eqnarray*}
for some $K>0$. 
Hence, $\overline b\geq n_2$. The same argument shows that $b\leq n_2$, and the claim is proved. 
Since $I_n$ cannot contain isolated points in $(2\ZZ)^2$ which are not in $I_{n-1},$ the proof is complete.  
\end{proof}

\section{The iteration procedure}\label{proc}
Given $L,L^\prime>0$ we define 
$$I_\e=([0, i_\e(L)]\times [0,  i_\e(L^\prime)]) \cap \mathbb N^2$$
where $i_\e(x)$ denotes for any $x>0$ the greater even integer 
less than $\lfloor \frac{x}{\e} \rfloor$;  
that is,
\begin{equation} \label{ie}
i_\e(x)=2\Big\lfloor\frac{\lfloor x/\e \rfloor}{2}\Big\rfloor. 
\end{equation}
Setting $I_\e^0=I_\e$, for $n\geq 1$ we denote by 
$I_\e^n$ a minimum point for the energy 
\begin{equation}\label{energia}
E_\e(I,I^{n-1}_\e)=F_\e(I)+\frac{1}{\tau}D_\e(I,I^{n-1}_\e).  
\end{equation}  
Setting for $(h,k)\in [0, i_\e(L^\prime)/4]\times [0, i_\e(L)/4]$ 
\begin{equation}\label{def-c} 
C_\e(h,k)=[2h, i_\e(L^\prime)-2h]\times[2k, i_\e(L)-2k], 
\end{equation} 
Lemma \ref{prop} ensures that 
if $I$ is a minimizer of (\ref{energia}) then there exist 
$(h^n_\e,k^n_\e)\in ([0, i_\e(L^\prime)/4]\times [0, i_\e(L)/4])\cap \mathbb N^2$ and 
$W^n_\e\subset \mathbb N^2$ with 
$C_\e(h^n_\e,k^n_\e)\cap \mathbb N^2\subseteq W^n_\e \subseteq I_\e$ 
such that 
\begin{equation}\nonumber 
I=I_\e(h^n_\e,k^n_\e,W^n_\e)=(C_\e(h^n_\e,k^n_\e) \cap \mathbb N^2)
\cup (W^n_\e\cap \mathbb (2 \mathbb N)^2).
\end{equation} 
Moreover, $h^n_\e\geq h_\e^{n-1}$, $k^n_\e\geq k_\e^{n-1}$ and $W^n_\e\subseteq W^{n-1}_\e$. 

In the following sections, we show that, up to a subset of $(2\mathbb N)^2$, the minimum problem 
for the energy $E_\e(I,I_\e)$ has a unique solution 
$$I_\e^1=(C_\e(h^1_\e,k^1_\e) \cap \mathbb N^2)
\cup (W^1_\e\cap \mathbb (2\mathbb N)^2)$$
with $h^1_\e$ and $k^1_\e$ independent of $\e$ for $\e$ small enough. 
Hence, the corresponding result holds for $I_\e^n$.   
In particular, we show that 
if the initial length $L$ of an edge is lower than 
a critical threshold then 
the corresponding ``displacement'' $h$ depends on the initial length as follows: 
$$h=h(L)=\Big\lfloor\frac{2\beta\gamma}{3L}-\frac{2\alpha\gamma}{3}+\frac{1}{6}\Big\rfloor.$$  
  
As to the set $W_\e^1$ of the {\em weak islands} (i.e., points with weak connections), 
note that the variation of the energy 
due to an isolated set $\{i\}\subset (2\mathbb N)^2$ 
is given by $$4\alpha\e^2-\frac{\e^3}{\tau}\dist_\infty(i, C(I_\e)).$$
Hence  we may rule out the formation of weak islands 
only if 
$4\alpha-\frac{\e}{\tau}\geq 0,$ that is if 
\begin{equation} 
4\alpha\gamma\geq 1. 
\end{equation}
Thus, we consider two different cases in dependence of the value of $\alpha\gamma$. 

\section{Case $4\alpha\gamma<1$} 
Recalling the properties shown in the previous section, 
it follows that in the minimum problem for the energy $E_\e(I,I_\e)$ defined in 
(\ref{energia}) 
we can consider only sets of the form 
\begin{equation}\label{setsie} 
I_\e(h,k)= 
\big(C_\e(h,k)\cap \mathbb N^2 \big)
\cup \big(I_\e \cap (2\mathbb N)^2\big) 
\end{equation} 
for $(h,k)\in ([0,\frac{L^\prime}{4\e}]\times [0,\frac{L}{4\e}])\cap \mathbb N^2,$ 
and the minimum problem for $E_\e$ corresponds to minimize 
\begin{equation} \nonumber 
f_\e(h,k)
=\frac{1}{\e}\big(F_\e(I_\e(h,k))-F_\e(I_\e)+\frac{1}{\tau}D_\e(I_\e(h,k), I_\e)\big).
\end{equation} 
Since we are interested in $\e\to 0$, it is not restrictive to consider the case when 
\begin{equation}\label{condo}
\e=\frac{\min\{L,L^\prime\}}{2n}\hbox{ for $n\in\mathbb N$ large enough}
\end{equation} 
(hence $i_\e(L)=\frac{L}{\e}$ if $L\leq L^\prime$ or $i_\e(L^\prime)=\frac{L^\prime}{\e}$ 
if $L\geq L^\prime$). We will then make this assumption, commenting on 
the error that we make under this hypothesis when necessary.

In the sequel of the section, we prove the following result. 
\begin{theorem} \label{teo-minimo} 
For $\e$ small enough, the energy $E_\e(I_\e(h,k),I_\e)$ 
has a unique minimum point in $([0,\frac{L^\prime}{4\e}]\times [0,\frac{L}{4\e}]) \cap \mathbb N^2$  
given by 
\begin{equation} \nonumber 
\left\{ \begin{array}{ll} \vspace{1mm}
(0,0) &  \hbox{ if } \min\{L,L^\prime\}=\lambda_c < \max\{L,L^\prime\}\\
(n(L),n(L^\prime)) &  \hbox{ otherwise }  
\end{array} 
\right.
\end{equation} 
where for any $l>0$ 
\begin{equation} 
\label{def-n}
n(l)=\left\{ \begin{array}{ll} \vspace{1mm}
0 &  \hbox{ if } l >\lambda_c \\
\big\lfloor\frac{2\beta\gamma}{3l}-\frac{2\alpha\gamma}{3}+\frac{1}{6}\big\rfloor &  \hbox{ if }  l\leq \lambda_c 
\end{array} 
\right.
\end{equation} 
and the critical threshold $\lambda_c$ is given by 
\begin{equation}\label{soglia-critica}
\lambda_c=\frac{4\beta\gamma}{4\alpha\gamma+5}. 
\end{equation}
\end{theorem}
\begin{proof}
Suppose $L^\prime\geq L$. 
First, we note that, in the case $L^\prime>L$, if $4\e h > L$ 
then the variation of the energy is strictly positive (for $\e$ small enough), 
so that the minimum point does not belong to this set and 
\begin{equation}\label{h-grande}
f_\e(h,k)>
\min_{[0,\frac{L}{4\e}]^2 \cap \mathbb N^2} f_\e 
\quad\ \ \hbox{ if } h> \frac{L}{4\e}.
\end{equation} 
Indeed, the dissipation term turns out to be larger than  
\begin{eqnarray*}
&&\frac{4\e^3}{\tau}\Biggl(\sum_{j=1}^{\lfloor L/(2\e)\rfloor }\biggl(\sum_{l=1}^j l 
+ j\Big(\Big\lfloor \frac{L}{2\e}\Big\rfloor-j\Big)\biggr) \\
&&
\qquad\qquad-\sum_{j=1}^{\lfloor L/(4\e)\rfloor}\biggl(\sum_{l=1}^j (2l-1) + (2j-1)\Big(\Big\lfloor\frac{L}{4\e}\Big\rfloor-j\Big)\biggr) 
\Biggr)\nonumber 
\end{eqnarray*}
and thus, up to a uniformly bounded term, larger than $\frac{L^3}{8\e\gamma}$. 
Since the variation of the boundary energy $F_\e$ is uniformly bounded, 
for $\e$ sufficiently small it follows that 
\begin{equation}
f_\e(h,k)\geq \e^{-2}C
\ \ \ \hbox{ if } \ L< 4\e h \leq L^\prime 
\nonumber 
\end{equation}
where $C>0$ is independent of $h,k,\e$, and (\ref{h-grande}) follows. 

Now, we consider $(h,k)\in [0,\frac{L}{4\e}]^2 \cap \mathbb N^2.$ 
Recalling that by assumption $\frac{L^\prime}{\e}\in 2\mathbb N$, 
the variation of the boundary energy is given by  
\begin{equation}\label{variazione-bordo}
\left. \begin{array}{ll}
\vspace{2mm}
F_\e(I_\e(h,k))-F_\e(I_\e)=&
4\e(-\beta h +\alpha Lh -\beta k + \alpha L^\prime k)\\
&+4\alpha\e^2( h+ k-4 hk) 
- 4\alpha\e^2 \varrho_\e k
\end{array}
\right.
\end{equation} 
where $\varrho_\e=\frac{L^\prime}{\e}-i_\e(L^\prime)\in [0,2)$. 
Note that if we do not assume (\ref{condo}) then we
have an additional term  $-4\alpha\e^2 \varrho_\e^L h$ with 
$\varrho_\e^L=\frac{L}{\e}-i_\e(L)$

As for the dissipation term, setting $M=\max\{h,k\}$ we have  
\begin{eqnarray}
&&D(I_{\e}(h,k),I_\e)=2 \e^3\Big(\Big(\frac{L}{\e}+1-4M\Big)
\sum_{j=1}^{2h}j-\Big(\frac{L}{2\e}+1-2M\Big)\sum_{j=1}^{h}(2j-1)\Big)\nonumber \\
&&\hspace{1cm}+\displaystyle 2\e^3 \Big(
(i_\e(L^\prime)+1-4M)\sum_{j=1}^{2k}j
-\Big(\frac{i_\e(L^\prime)}{2}+1-2M\Big)\sum_{j=1}^{k}(2j-1)\Big)
\nonumber \\
&&\hspace{1cm}+4\e^3 \Big( \sum_{j=1}^{2M}\big(\sum_{l=1}^j l +j(2M-j)\big) 
-  \sum_{j=1}^{M}\big(\sum_{l=1}^j (2l-1) +(2j-1)(M-j)\big) 
\Big).\nonumber 
\end{eqnarray}
Hence, 
\begin{eqnarray}
\frac{1}{\tau}D(I_{\e}(h,k),I_\e)&=&\displaystyle \frac{\e^3}{\tau} 
\Big(\frac{L}{\e}h(3h+2)+2h^2+2h
-12h^2M-8hM\Big)
\nonumber \\
&&
+\frac{\e^3}{\tau} \Big( \frac{L^\prime}{\e}k(3k+2) +2k^2
+2k
-12k^2M-8kM\Big)
\nonumber \\
&&
-\frac{\e^3}{\tau}
(3\varrho_\e k^2+2\varrho_\e k)
+\frac{\e^3}{\tau} (8M^3+8M^2) 
\nonumber \\
&=&\displaystyle \frac{\e}{\gamma} \big(Lh(3h+2)+ L^\prime k(3k+2)\big) 
\nonumber \\
&&
 +\frac{\e^2}{\gamma}(2h^2
 +2h
 +2k^2
 +2k-8hk
 -12M\min\{h,k\}^2-4M^3)
\nonumber \\
&&
-\frac{\e^2}{\gamma}
\varrho_\e k(3k+2)
.
\nonumber
\end{eqnarray} 
Again, if we do not assume (\ref{condo}) we obtain an additional term
$-\frac{\e^2}{\gamma}\varrho_\e^L h(3h+2)$ with
$\varrho_\e^L=\frac{L}{\e}-i_\e(L)$.

Setting 
\begin{equation}\label{def-pi}
\pi(x)=3x^2+2(2\alpha\gamma+1)x
\end{equation}
and, for $l>0$,  
\begin{equation}\label{def-PL}
P_l(x)=
\frac{l}{\gamma}\pi(x)-4\beta x
=\frac{3l}{\gamma}x^2 -2\Big(2\beta-\Big(2\alpha+\frac{1}{\gamma}\Big)l\Big)x
\end{equation}
we get \begin{equation}\label{fehk}
f_\e(h,k)
=P_L(h)+P_{L^\prime}(k)+\frac{\e}{\gamma} R(h,k) 
-\frac{\e}{\gamma} \varrho_\e\pi(k) 
\nonumber 
\end{equation}
where $R(h,k)$ is the symmetric function defined for $h\geq k$ by 
\begin{equation}\label{def-R}
\left. 
\begin{array}{ll}
R(h,k)=
\vspace{2mm}
&\displaystyle 2h^2+2(2\alpha\gamma+1)h+2k^2+2(2\alpha\gamma+1)k
\\
&-8(2\alpha\gamma+1)hk-12h k^2-4h^3.
\end{array}
\right.
\end{equation} 
Without assuming (\ref{condo}) we get an additional term
$-\frac{\e}{\gamma}
\varrho_\e^L \pi(h)$ in (\ref{fehk}).
\color{black} 
\medskip 

The sequence $\{f_\e\}$ uniformly converges   
on the compact sets of $\mathbb R^2$ to $f$ defined by 
$f(x,y)=P_L(x)+P_{L^\prime}(y).$ We 
denote by $(m(L),m(L^\prime))$ 
the minimum point of $f$ in $\mathbb R^2$, so that $m(l)$ is defined by 
\begin{equation}\label{def-mL}
m(l)=\frac{2\beta\gamma-(2\alpha\gamma+1)l}{3l}.
\end{equation}
Note that if $L=L^\prime$ the function $f_\e$ is symmetric. If $L^\prime>L$, then 
the symmetry of $R$ gives 
$f_\e(x,y)-f_\e(y,x)=(P_L-P_{L^\prime}-\frac{\e\varrho_\e}{\gamma} \pi)(x)-
(P_L-P_{L^\prime}-\frac{\e\varrho_\e}{\gamma} \pi)(y);$ 
since $P_L-P_{L^\prime}-\frac{\e\varrho_\e}{\gamma}\pi$ 
is strictly decreasing in $(0,+\infty)$ for $\e$ small enough, 
it follows that 
\begin{equation}\label{sopra-no}
f_\e(x,y)>f_\e(y,x) \quad \hbox{ if }\ 0<x<y. 
\end{equation}
Hence the minimum of $f_\e$ in $[0,\frac{L}{4\e}]^2$ is achieved in $\{(x,y): x\geq y\}$.

Now we prove 
that the minimum of $f_\e$ is in fact achieved in a compact set independent of $\e$. 
\begin{proposition}   \label{limitato}
There exists $\overline x>0$ independent of $\e$ such that, for $\e$ small enough,  
if $(x_\e,y_\e)$ is a minimum point of $f_\e$ in $[0,\frac{L}{4\e}]^2,$ 
then $(x_\e,y_\e)\in [0,\overline x]^2$. 
\end{proposition} 
\begin{proof} 
Since $f_\e$ coincides in $\{x>y\}$   
with a polynomial function of degree $3$, then if $L<L^\prime$ (hence $m(L)>m(L^\prime)$) the 
uniform convergence ensures that the (unique) critical minimum point of $f_\e$ belongs to 
a compact neighborhood of $(m(L),m(L^\prime))$ included in $\{x>y\}$, and independent of $\e$.   
In the case $L=L^\prime$, 
the computation of the partial derivatives gives 
$\nabla \!f_\e\neq (0,0)$ in 
$[0,\frac{L}{4\e}]^2\cap\{x>y\}$. 

We fix $\overline x>\max\{m(L)+1,0\}$; the minimum of $f_\e$ in $[0,\frac{L}{4\e}]^2\cap\{ x\geq \overline x, \ x\geq y\}$ 
is then achieved on the boundary,  
where a computation shows that 
$$f_\e(x,y)\geq P_L(\overline x)+\min P_{L^\prime}+o(1)_{\e\to 0}.$$
Choosing $\overline x$ such that $P_L(\overline x)+\min P_{L^\prime}>0,$ 
then in $[0,\frac{L}{4\e}]^2\cap\{ x\geq \overline x, \ x\geq y\}$ we have 
$f_\e(x,y)>0=f_\e(0,0)$.  
Since $f_\e(x,y)\geq f_\e(y,x)$ if $x\leq y$, 
the thesis follows.  
\end{proof}
To prove Theorem \ref{teo-minimo}, 
we have to consider the minimum problem for $f_\e$ 
in $([0,\frac{L^\prime}{4\e}]\times[0,\frac{L}{4\e}])\cap \mathbb N^2$. Recalling $(\ref{h-grande})$, the minimum is 
achieved in $[0,\frac{L}{4\e}]^2\cap \mathbb N^2$, and 
thanks to Proposition \ref{limitato}, 
it is sufficient to show that 
$f_\e$ has a unique minimum point independent of $\e$ in $[0,\overline x]^2 \cap \mathbb N^2.$

If $f$ has a unique minimum point in $[0,\overline x]^2 \cap \mathbb N^2,$ then the uniform convergence 
ensures that, for $\e$ small enough, 
it coincides with the (unique) minimum point of $f_\e$ in $[0,\overline x]^2 \cap \mathbb N^2,$ 
concluding the proof of the uniqueness and independence on $\e$.  

The minimum of $f$ is achieved in the points $(h,k)$ in $[0,\overline x]^2 \cap \mathbb N^2$ 
which minimize 
the distance 
from $(m(L), m(L^\prime))$. 
Thus, the uniqueness fails only if $m(L)+\frac{1}{2}$ 
or $m(L^\prime)+\frac{1}{2}$ belongs to $\mathbb N\setminus\{0\}$. In this case, 
we have to compare the values of $f_\e-f$ in the set of the minimum points of $f$. 

Note that the condition $m(L)+\frac{1}{2}<1$ corresponds to $L>\lambda_c$, where $\lambda_c$ 
is the critical length defined in (\ref{soglia-critica}). 
Hence, recalling 
that for $l>0$ 
$$m(l)+\frac{1}{2}=\frac{2\beta\gamma}{3l}-\frac{2\alpha\gamma}{3}+
\frac{1}{6},$$  
if both $m(L)+\frac{1}{2}$ 
and $m(L^\prime)+\frac{1}{2}$ do not belong to $\mathbb N\setminus\{0\}$ then the 
minimum point of $f_\e$ in $[0,\overline x]^2 \cap \mathbb N^2$ is unique and it is given by 
\begin{enumerate}
\item[-] $(0,0)=(n(L), n(L^\prime))$ if $L>\lambda_c$;
\item[-] $(\lfloor m(L)+\frac{1}{2}\rfloor, 0)=(n(L), n(L^\prime))$ if 
$L<\lambda_c<L^\prime$; 
\item[-] $(\lfloor m(L)+\frac{1}{2}\rfloor, \lfloor m(L^\prime)+\frac{1}{2}\rfloor)=
(n(L), n(L^\prime))$ if $L^\prime<\lambda_c$ 
\end{enumerate} 
where the function $n$ is defined in (\ref{def-n}).

It remains to check the case when $m(L)+\frac{1}{2}$ 
or $m(L^\prime)+\frac{1}{2}$ belongs to $\mathbb N\setminus\{0\}$. 
Note that  
$R(x,y)-\varrho_\e\pi(y)$ decreases with respect to each variable in $\{x,y\geq 1\}$, and 
also $R(x,0)$ is strictly decreasing in $\{x\geq 1\}$; moreover, 
for any $x\geq 1$ then $R(x,0)>R(x,1)-\varrho_\e \pi(1)$. 
The monotonicity properties of $R(x,y)-\varrho_\e \pi(y)$ allow to deduce that 
the minimum is achieved in $(\lfloor m(L)+\frac{1}{2}\rfloor, \lfloor m(L^\prime)+\frac{1}{2}\rfloor)=
(n(L), n(L^\prime))$ 
for each $L,L^\prime$ except when $L=\lambda_c$. 
In this case, if also $L^\prime=\lambda_c$, the comparison between $f_\e(0,0),$ $f_\e(1,0),$ 
$f_\e(0,1)$ and $f_\e(1,1)$ ensures that the minimum point is $(1,1)$, again corresponding 
to $(\lfloor m(L)+\frac{1}{2}\rfloor, \lfloor m(L^\prime)+\frac{1}{2}\rfloor)$. 
If $L=\lambda_c$ and $L^\prime>\lambda_c$, 
since $f_\e(1,0)=4\e\alpha>f_\e(0,0),$ 
then 
the minimum is obtained in $(0,0)=(\lfloor m(L)-\frac{1}{2}\rfloor, \lfloor m(L^\prime)+\frac{1}{2}\rfloor)$.


This completes the proof of Theorem \ref{teo-minimo}.
\end{proof}

In particular, except for the case when one or both the initial lengths 
are greater than $\lambda_c$, 
the minimizing set for the energy $E_\e$ is 
$$I_\e(n(L),n(L^\prime))=I_\e\Big(\Big\lfloor\frac{2\beta\gamma}{3L}-\frac{2\alpha\gamma}{3}+\frac{1}{6}\Big\rfloor,
\Big\lfloor\frac{2\beta\gamma}{3L^\prime}-\frac{2\alpha\gamma}{3}+\frac{1}{6}\Big\rfloor\Big)$$
with $I_\e$ defined in (\ref{setsie}).

We can use the set $I_\e(n(L),n(L^\prime))$ as a recursive datum for the energy (\ref{energia}).
Note that the weak sites of the initial configurations will be part of each minimizers. 
We can then describe the motion through the velocity of the moving rectangle corresponding
to $C_\e(h,k)$.

 \color{black}

\color{black} 
We finally compute the velocity of a side.
Consider an edge with initial non-scaled length $L$. We proved that it moves only 
if $L<\lambda_c$ (also for $L=\lambda_c$ if $L$ is not the minimal length of the edges). 
In this case, the displacement of the edge is given by 
\begin{equation} 
2\Big\lfloor\frac{2\beta\gamma}{3L}-\frac{2\alpha\gamma}{3}+\frac{1}{6}\Big\rfloor.
\end{equation}
Hence, if $L<\lambda_c$, 
$$v_\gamma(L)=-\frac{2}{\gamma}\Big\lfloor\frac{2\beta\gamma}{3L}-\frac{2\alpha\gamma}{3}+\frac{1}{6}\Big\rfloor 
=-\frac{4\beta}{3L}+\frac{4\alpha}{3}+\frac{6r(L)-1}{3\gamma}
$$ 
where $r(L)=m(L)+\frac{1}{2}-\lfloor m(L)+\frac{1}{2} \rfloor$. 

\section{Case $4\alpha\gamma\geq 1$} 
The properties of the energy shown in Section \ref{proc} 
imply that a weak island may appear only when the distance from the boundary is strictly greater 
than $4\alpha\gamma$. 
Hence, except if $4\alpha\gamma$ is an odd integer, 
an isolated point $i\in (2\mathbb N)^2$ belongs to a minimizing set for $E_\e(I,I_\e)$  
if and only if the distance from the boundary of $I_\e$ is greater than 
the minimal odd integer larger than $\lfloor 4\alpha\gamma\rfloor$, 
which is given by $2N_{\alpha\gamma}-1$ where  
$$N_{\alpha\gamma}=\Big\lfloor\frac{\lfloor4\alpha\gamma\rfloor+1}{2}\Big\rfloor.$$ 
Recalling the definition (\ref{def-c}) 
of $C_\e(s,t)$, if we define for $(s,t)\in[0,i_\e(L^\prime)/4]\times[0,i_\e(L)/4]$
\begin{equation}\label{def-j} 
J_\e(s,t)=
\Big( C_\e(s,t) \cap \mathbb N\Big)
\cup\Big(
C_\e(N_{\alpha\gamma},N_{\alpha\gamma})\cap (2\mathbb N)^2\Big)
\end{equation} 
where we omit the dependence on $L$ and $L^\prime$, then 
if $4\alpha\gamma$ is an odd integer it follows that 
\begin{equation}\label{caso-limite}
E_\e(J_\e(s,t), I_\e)=E_\e\big((C_\e(s,t) \cap \mathbb N)\cup(W\cap (2\mathbb N)^2), I_\e)
\end{equation}
where $W\subseteq C_\e(N_{\alpha\gamma}-1,N_{\alpha\gamma}-1)$. 
Hence, except if $4\alpha\gamma$ is an odd integer, 
the minimum problem for the energy $E_\e(I,I_\e)$ corresponds to minimize 
\begin{equation}
g_\e^{\alpha\gamma}(s,t)=\frac{1}{\e}(E_\e(J_\e(s,t),I_\e)-E_\e(I_\e))  
\end{equation} 
in $([0,i_\e(L^\prime)/4]\times[0,i_\e(L)/4])\cap \mathbb N^2$. 

In the following theorem we give an explicit computation of the set of minimum points of
$g^{\alpha\gamma}_\e(s,t)$, 
showing that this set is independent of $\e$ (for $\e$ small enough), and that 
there is uniqueness 
except if $4\alpha\gamma=2$ and 
$\min\{L,L^\prime\}=\lambda^+< \max\{L,L^\prime\}$.   
We prove the following result, which takes care of almost all values of $4\alpha\gamma$,
except when this is an odd integer. In that case we may have the non uniqueness 
phenomenon as described in (\ref{caso-limite}), which will not affect the final description of the motion.

\begin{theorem}
\label{teo-n} 
Let $4\alpha\gamma\in{\mathbb R}\setminus (2\NN+1)$ (i.e., not an odd integer).
If $4\alpha\gamma>2$, then the 
unique minimum point of $g^{\alpha\gamma}_\e(s,t)$ in the set 
$([0,i_\e(L^\prime)/4]\times[0,i_\e(L)/4])\cap \mathbb N^2$ is 
$(\varphi(L), \varphi(L^\prime))$ where 
\begin{equation} 
\varphi(l)=\left\{ \begin{array}{ll} \vspace{1mm}
0 & \hbox{ if } l> \lambda^+ \\ \vspace{1mm}
\lfloor \frac{\beta\gamma}{2l}+\frac{1}{4}\rfloor 
& \mbox{ if } l\in (\lambda^-,\lambda^+] \\ \vspace{1mm}
N_{\alpha\gamma} & \hbox{ if } l\in (\lambda_c^\ast,\lambda^-] \\ 
\lfloor\frac{2\beta\gamma}{3l}-\frac{2\alpha\gamma}{3}+\frac{1}{6}\rfloor & \hbox{ if } l\leq \lambda_c^\ast 
\end{array}
\right. 
\end{equation} 
and the critical thresholds are given by  
\begin{equation} \label{soglie-n} 
\lambda_c^\ast=\frac{4\beta\gamma}{4\alpha\gamma+5+6N_{\alpha\gamma}}, \quad 
\lambda^-=\frac{2\beta\gamma}{4N_{\alpha\gamma}-1} \quad \hbox{ and } \quad 
\lambda^+= \frac{2\beta\gamma}{3}. 
\end{equation}

\medskip 
\noindent If $4\alpha\gamma<2$, the minimum point is again unique, and it is given by 
$$
\left\{  
\begin{array}{ll} \vspace{2mm}
(0,0) & \hbox{ if } \min\{L,L^\prime\}=\lambda^+< \max\{L,L^\prime\}\\
(\varphi(L), \varphi(L^\prime)) & \hbox{ otherwise. } 
\end{array}
\right.$$ 

\medskip 
\noindent If $4\alpha\gamma=2$, for $\min\{L,L^\prime\}=\lambda^+< \max\{L,L^\prime\}$ 
there is no uniqueness, and the minimum points are $(0,0)$ and $(\varphi(L), \varphi(L^\prime))=(1,0)$. 
Otherwise, the minimum point is unique and it is given by $(\varphi(L), \varphi(L^\prime))$. 
\end{theorem}

\begin{remark}\rm
(i) Note that the non-uniqueness exactly for $4\alpha\gamma=2$ depends on the
simplifying choice (\ref{condo}). This choice however influences only the value 
at which we have non uniqueness, but not the final description.

(ii) Note that $\varphi$ can be equivalently written as the integer part in (\ref{molo}).
\end{remark}

\begin{proof} 
We assume $L^\prime\geq L$. 
Since the weak islands appear only when the distance from the boundary is greater than $2N_{\alpha\gamma}$, 
in the computation of the minimum point of the energy
we have to consider three cases in dependence on the values of $s$ and $t$, namely $s,t\leq N_{\alpha\gamma}$, 
$\min\{s,t\}\leq N_{\alpha\gamma}\leq \max\{s,t\}$ and $s,t\geq N_{\alpha\gamma}$.

\medskip 

We start by computing the expression of $g_\e^{\alpha\gamma}$ 
in $[0,N_{\alpha\gamma}]^2\cap \mathbb N^2$. 

\begin{figure}[h!]
\centerline{\includegraphics [width=4.5in]{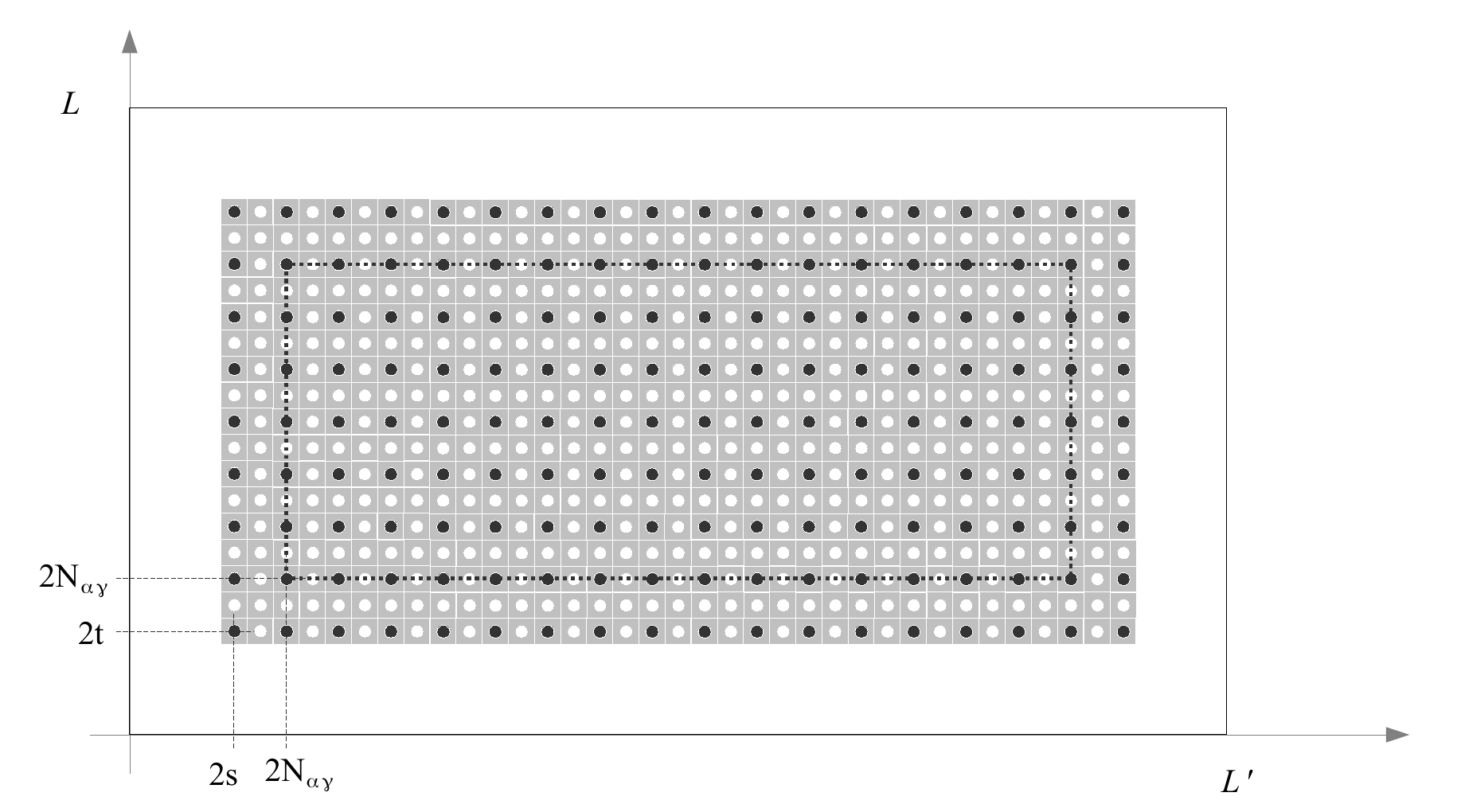}}
\caption{The set $J_\e(s,t)$ with $s,t\leq N_{\alpha\gamma}$}\label{Jepsilon1}
   \end{figure}
Assuming $s\geq t$, for  
$(s,t)\in[0,N_{\alpha\gamma}]^2\cap \mathbb N^2$ 
we get
\begin{eqnarray}
g_\e^{\alpha\gamma}(s,t)&=& -4\beta s - 4\beta t
-4\e\alpha s -4\e\alpha t +\frac{4\e}{\gamma}\sum_{j=1}^{2s}(\sum_{l=1}^j l+ j(2s-j)) \nonumber \\
&&+ 
\frac{2\e}{\gamma}(i_\e(L)+1-4s)\sum_{j=1}^{2s} j
+\frac{2\e}{\gamma}(i_\e(L^\prime)+1-4s)\sum_{j=1}^{2t} j \nonumber \\
&=&  
\frac{L}{\gamma}p(s)- 4\beta s
+\frac{L^\prime}{\gamma}p(t)-4\beta t-\frac{\e}{\gamma}\varrho_\e p(t)\nonumber \\
&&+ \frac{\e}{\gamma}\Big(\frac{4}{3}(1+4s)p(s)
+(1-4s) (p(s)+p(t)))-4\alpha\gamma (s+t)\Big), 
 \nonumber 
\end{eqnarray}
where $p(x)=2x(2x+1)$ and $\varrho_\e=\frac{L^\prime}{\e}-i_\e(L^\prime)$.  
Hence, 
setting for $l>0$ 
\begin{equation}\label{polinomio-piccolo} 
Q_l(x)=\frac{l}{\gamma}p(x)-4\beta x=
\frac{4 l}{\gamma}x^2-2\Big(2\beta-\frac{l}{\gamma}\Big)x, 
\end{equation} 
the function $g_\e^{\alpha\gamma}$ 
can be expressed in $[0,N_{\alpha\gamma}]^2\cap \mathbb N^2$ as 
\begin{equation}\label{g-piccola}
g_\e^{\alpha\gamma}(s,t)=Q_L(s)+Q_{L^\prime}(t)+\frac{\e}{\gamma} r(s,t)-\frac{\e}{\gamma}\varrho_\e p(t) 
\end{equation}
where 
$r(s,t)$ is 
the symmetric function 
defined for $s\geq t$ 
by 
\begin{equation}
r(s,t)=\frac{4}{3}(1+4s)p(s)
+(1-4s) (p(s)+p(t))-4\alpha\gamma (s+t). \nonumber 
\end{equation}
The following lemma holds. 
\begin{lemma} \label{minore} 
If $4\alpha\gamma>2$, 
the unique minimum point of $g^{\alpha\gamma}_\e(s,t)$ in the set 
$[0,N_{\alpha\gamma}]^2 \cap \mathbb N^2$ is 
given by 
$(\varphi(L)\wedge N_{\alpha\gamma}, \varphi(L^\prime)\wedge N_{\alpha\gamma})$.  

\noindent If $4\alpha\gamma<2$, the minimum point is again unique, and it is given by 
$$
\left\{  
\begin{array}{ll} \vspace{2mm}
(0,0) & \hbox{ if } 
L=\lambda^+< 
L^\prime\\
(\varphi(L)\wedge N_{\alpha\gamma},\varphi(L^\prime)\wedge N_{\alpha\gamma}) 
& \hbox{ otherwise. } 
\end{array}
\right.$$ 

\medskip 
\noindent If $4\alpha\gamma=2$, for $L=\lambda^+<L^\prime$ 
there is no uniqueness, and the minimum points are $(0,0)$ and 
$(\varphi(L)\wedge N_{\alpha\gamma},\varphi(L^\prime)\wedge N_{\alpha\gamma})=(1,0)$.
Otherwise, the minimum point is unique and it is given by 
$(\varphi(L)\wedge N_{\alpha\gamma},\varphi(L^\prime)\wedge N_{\alpha\gamma}).$
\end{lemma}
\begin{proof} 
The result follows by the uniform convergence of $\e r(s,t)-\e \varrho_\e p(t)$ to $0$ in $[0,N_{\alpha\gamma}]^2$, so that 
the minimum of the energy is obtained in the set of the points in $[0,N_{\alpha\gamma}]^2\cap \mathbb N^2$ 
minimizing the distance 
from the minimum point of $Q_L(s)+Q_{L^\prime}(t)$, given by 
$(\mu(L),\mu(L^\prime))$,   
where 
\begin{equation}\label{mu-L}
\mu(l)=\frac{2\beta\gamma-l}{4l}. 
\end{equation}

This set contains only one point except if $\mu(L)+\frac{1}{2}$ or 
$\mu(L^\prime)+\frac{1}{2}$ belong to the set of integers $\{1,\ldots, N_{\alpha\gamma}-1\}$. 
In all other cases, noting that the conditions $\mu(l)=N_{\alpha\gamma}-\frac{1}{2}$ 
and $\mu(l)=\frac{1}{2}$ 
give the critical lengths $\lambda^-$ and $\lambda^+$ respectively, 
the thesis follows. 

When the set of minimum points of $Q_L(s)+Q_{L^\prime}(t)$ in $[0,N_{\alpha\gamma}]^2\cap \mathbb N^2$ 
contains more than one point, 
we need some monotonicity properties of $r$. The function  
$r(s,t)-\varrho_\e p(t)$ is strictly decreasing with respect to each variable in $\{s,t\geq 1\}$; moreover,  
in $\{s\geq 1\}$ the function 
$r(s,0)$ is strictly decreasing and $r(s,0)>r(s,1)-\varrho_\e p(1)$. 
This gives the thesis except for the case 
$\mu(L)=\frac{1}{2}$ and $\mu(L^\prime)\leq\frac{1}{2}$ 
(hence $L=\lambda^+$ and $L^\prime\geq \lambda^+$).   
By 
noting that 
$$r(1,1)-\varrho_\e p(1)<\min\{r(1,0); r(0,1)-\varrho_\e p(1); r(0,0)\}, \quad r(1,0)=2-4\alpha\gamma$$
the proof is complete .
\end{proof}

\noindent 
Now, we compute the expression of $g_\e^{\alpha\gamma}$ in 
$([N_{\alpha\gamma}, \frac{L^\prime}{4\e}]\times[N_{\alpha\gamma}, \frac{L}{4\e}])\cap \mathbb N^2$. 
 
Following the proof of Theorem \ref{teo-minimo}, it turns out that it is not restrictive to consider 
only 
$(s,t)\in [N_{\alpha\gamma},\frac{L}{4\e}]^2 \cap \mathbb N^2.$

We can decompose $g_\e^{\alpha\gamma}(s,t)$ in the sum of 
two contributions. The first is due to the set $J_\e(N_{\alpha\gamma},N_{\alpha\gamma})$ 
and it is given by (\ref{g-piccola})  
\begin{equation}\nonumber 
g_\e^{\alpha\gamma}(N_{\alpha\gamma},N_{\alpha\gamma})=Q_L(N_{\alpha\gamma})+Q_{L^\prime}(N_{\alpha\gamma})+\frac{\e}{\gamma} r(N_{\alpha\gamma},N_{\alpha\gamma})
-\frac{\e}{\gamma}\varrho_\e p(N_{\alpha\gamma}). 
\end{equation}
\begin{figure}[h!]
\centerline{\includegraphics [width=4.5in]{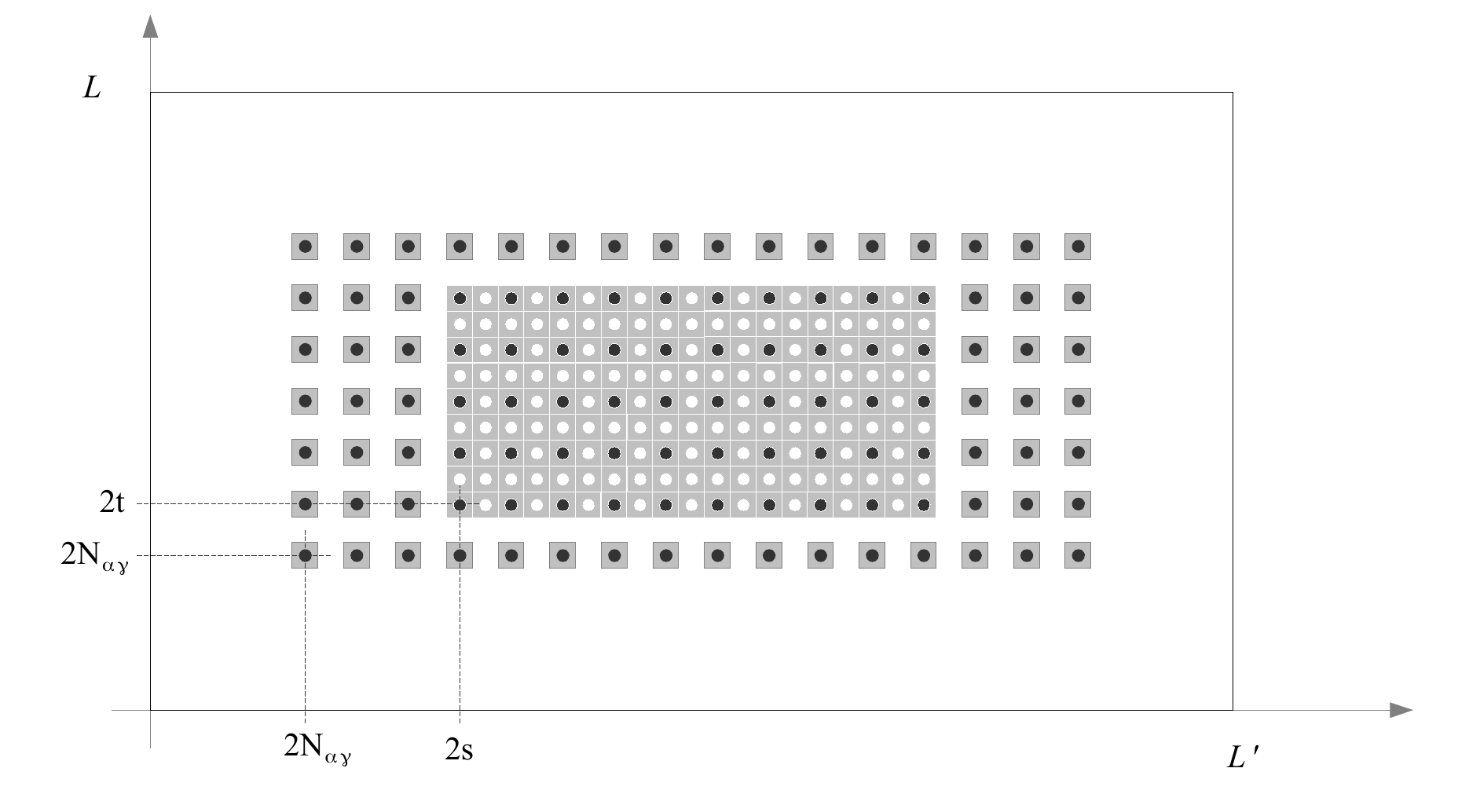}}
\caption{The set $J_\e(s,t)$ with $s,t\geq N_{\alpha\gamma}$}\label{Jepsilon2}
   \end{figure}
   
Then, we have a contribution which 
can be computed as in the case $4\alpha\gamma<1$ by substituting 
the initial set $I_\e$ with $I^{\alpha\gamma}_\e= 
C_\e(N_{\alpha\gamma},N_{\alpha\gamma})\cap \mathbb N^2$.   
Hence, 
following the proof of Theorem \ref{teo-minimo}, 
the contribution to $g_\e^{\alpha\gamma}$ is given by 
\begin{equation*} 
f_\e^{\alpha\gamma}(s-N_{\alpha\gamma}, t-N_{\alpha\gamma})+\frac{\e}{\gamma} \ 2N_{\alpha\gamma}\ \# 
(I_\e^{\alpha\gamma}\setminus J_\e(s,t)), 
\end{equation*} 
where the term $2N_{\alpha\gamma}\ \# 
(I_\e^{\alpha\gamma}\setminus J_\e(s,t))$ takes into account 
the additional distance $2N_{\alpha\gamma}$ of each point of the set from $C(I_\e)$, and 
\begin{equation}\label{feag}
\left. \begin{array}{ll}
f_\e^{\alpha\gamma}(h,k)&
=\disp \frac{1}{\e}\big(E_\e(J_\e(h+N_{\alpha\gamma}, k+N_{\alpha\gamma}), I_\e^{\alpha\gamma})
-F_\e(I_\e^{\alpha\gamma})\\
&=
\disp P_{L-4\e N_{\alpha\gamma}}(h)+P_{L^\prime-4\e N_{\alpha\gamma}}(k)
+\frac{\e}{\gamma}R(h,k)-\frac{\e}{\gamma}\varrho_\e \pi(k),
\end{array}
\right. 
\end{equation}
with $\pi$, $P_{l}$ and $R$  defined as in (\ref{def-pi}), (\ref{def-PL}) and (\ref{def-R}) respectively. 
Since 
\begin{equation*}\nonumber 
P_{L-4\e N_{\alpha\gamma}}(h)=
P_L(h)
-\frac{\e}{\gamma}(12N_{\alpha\gamma}h^2+
8N_{\alpha\gamma}(2\alpha\gamma+1)h), 
\end{equation*}
we get for $g_\e^{\alpha\gamma}(s,t)$ the following expression 
\begin{eqnarray*}
g_\e^{\alpha\gamma}(s,t)&=&P_{L}(s-N_{\alpha\gamma})
+\frac{6L N_{\alpha\gamma}}{\gamma}(s-N_{\alpha\gamma})
+Q_{L}(N_{\alpha\gamma})\\
&&+P_{L^\prime}(t-N_{\alpha\gamma})
+\frac{6L^\prime N_{\alpha\gamma}}{\gamma}(t-N_{\alpha\gamma})
+Q_{L^\prime}(N_{\alpha\gamma})\nonumber \\
&&+\frac{\e}{\gamma}r(N_{\alpha\gamma},N_{\alpha\gamma})+\frac{\e}{\gamma} R_{\alpha\gamma}(s-N_{\alpha\gamma},t-N_{\alpha\gamma})-\frac{\e}{\gamma}\varrho_\e \pi_{\alpha\gamma}(t)\nonumber 
\end{eqnarray*}
where $\pi_{\alpha\gamma}(k)=
 3k^2 
+2(3N_{\alpha\gamma}+2\alpha\gamma+1)k
+p(N_{\alpha\gamma})$ and 
\begin{eqnarray*}\nonumber 
&R_{\alpha\gamma}(h,k)=&
R(h,k)
+4N_{\alpha\gamma}(1-6N_{\alpha\gamma})(h+k)
-24N_{\alpha\gamma}hk\\
&&-12N_{\alpha\gamma}(h^2+k^2) 
-8N_{\alpha\gamma}(2\alpha\gamma+1)(h+k).
\end{eqnarray*}

Now we prove the following lemma. 
\begin{lemma} \label{maggiori}
The function $g^{\alpha\gamma}_\e(s,t)$ has a unique minimum point 
in the set 
$([N_{\alpha\gamma},\frac{L^\prime}{4\e}] 
\times [N_{\alpha\gamma},\frac{L}{4\e}]) \cap \mathbb N^2$,  
given by $(\varphi(L)\vee N_{\alpha\gamma},\varphi(L^\prime)\vee N_{\alpha\gamma})$.   
\end{lemma}
\begin{remark}
{\rm Note that $\varphi(l)\vee N_{\alpha\gamma}=\lfloor\frac{2\beta\gamma}{3l}-\frac{2\alpha\gamma}{3}+\frac{1}{6}\rfloor$ for $\lambda\leq \lambda_c^\ast$ and $\varphi(l)\vee N_{\alpha\gamma}=N_{\alpha\gamma}$ 
otherwise, 
where $\lambda_c^\ast$ is defined by (\ref{soglie-n}).} 
\end{remark}
\begin{proof}[Proof of Lemma {\rm\ref{maggiori}}.] 
We note that the minimum point of $P_{l}(x)
+\frac{6l N_{\alpha\gamma}}{\gamma}(x)$ is given by $m(l)-N_{\alpha\gamma}$, where 
$m(l)=\frac{2\beta\gamma}{3l}-\frac{2\alpha\gamma}{3}-\frac{1}{3}$ 
denotes the minimum point of $P_l$; 
the condition $m(l)-N_{\alpha\gamma}=\frac{1}{2}$ introduces the critical threshold 
$\lambda_c^\ast$ given by (\ref{soglie-n}). 

As in 
Proposition \ref{limitato}, 
we can prove that also in this case the minimum of $g_\e^{\alpha\gamma}$ is in fact achieved 
in a compact set independent of $\e$. 
This allows to show that 
a minimum point of $g_\e^{\alpha\gamma}$ in 
$[N_{\alpha\gamma},\frac{L}{4\e}]^2 \cap \mathbb N^2$ 
necessarily belongs to the set of points minimizing the distance from 
$(m(L), m(L^\prime))$.

Since the function $R_{\alpha\gamma}(x,y)-\varrho_\e\pi_{\alpha\gamma}(y)$ is strictly decreasing with respect to each variable 
in $\{x,y\geq 0\}$, then 
the thesis of Lemma \ref{maggiori} follows also if the set of points 
in $[N_{\alpha\gamma},\frac{L}{4\e}]^2 \cap \mathbb N^2$ 
minimizing the distance 
from $(m(L), m(L^\prime))$ 
contains more than one point.  
Note that the minimum is obtained for 
$(\lfloor m(L)+\frac{1}{2}\rfloor, 0)$ 
even if $L=\lambda_c^\ast$ and $L^\prime>\lambda_c^\ast.$
\end{proof}

Now, we have to consider the case $\min\{s,t\}\leq N_{\alpha\gamma}\leq \max\{s,t\}.$ 

If $(s,t)\in ([N_{\alpha\gamma},\frac{L^\prime}{4\e}]\times[0,N_{\alpha\gamma}])\cap \mathbb N^2$,  
we can decompose $g_\e^{\alpha\gamma}(s,t)$ in the sum of 
three contribution. The first is due to the set $J_\e(N_{\alpha\gamma},t)$ and it is given by (\ref{g-piccola})  
\begin{equation}\nonumber 
g_\e^{\alpha\gamma}(N_{\alpha\gamma},t) =Q_L(N_{\alpha\gamma})+Q_{L^\prime}(t)+\frac{\e}{\gamma} r(N_{\alpha\gamma},t)
-\frac{\e}{\gamma}\varrho_\e p(t). 
\end{equation}
The second term can be computed following the same argument of the case $s,t\geq N_{\alpha\gamma}$, and it is given by 
\begin{equation*} 
f_\e^{\alpha\gamma}(s-N_{\alpha\gamma}, 0)+\frac{\e}{\gamma} \ 2N_{\alpha\gamma} 
(3(s-N_{\alpha\gamma}) i_\e(L) +2(s-N_{\alpha\gamma})(1-6N_{\alpha\gamma})) 
\end{equation*}
where $f_\e^{\alpha\gamma}$ is defined as in (\ref{feag}). 
Moreover, we have to consider the dissipation due to 
$J_\e(N_{\alpha\gamma},t)\setminus J_\e(s,t)$; 
denoting by $Q_\e(s,t)$ the set $[2N_{\alpha\gamma}, 2s)\times [2t,2N_{\alpha\gamma})$, 
the additional term turns out to be 
\begin{eqnarray}
&\disp\frac{4\e}{\gamma}\sum_{Q_\e(s,t)\cap I_\e}\hspace{-3mm}\dist_{\infty}(i,C(I_\e))&=
\frac{4\e}{\gamma}(2s-2N_{\alpha\gamma}) \disp\sum_{j=1}^{2N_{\alpha\gamma}-2t}(j+2t) \nonumber \\
&&=\frac{8\e}{\gamma}(s-N_{\alpha\gamma})  (N_{\alpha\gamma}-t)(2N_{\alpha\gamma}+2t+1). \nonumber 
\end{eqnarray}
\begin{figure}[h!]
\centerline{\includegraphics [width=4.5in]{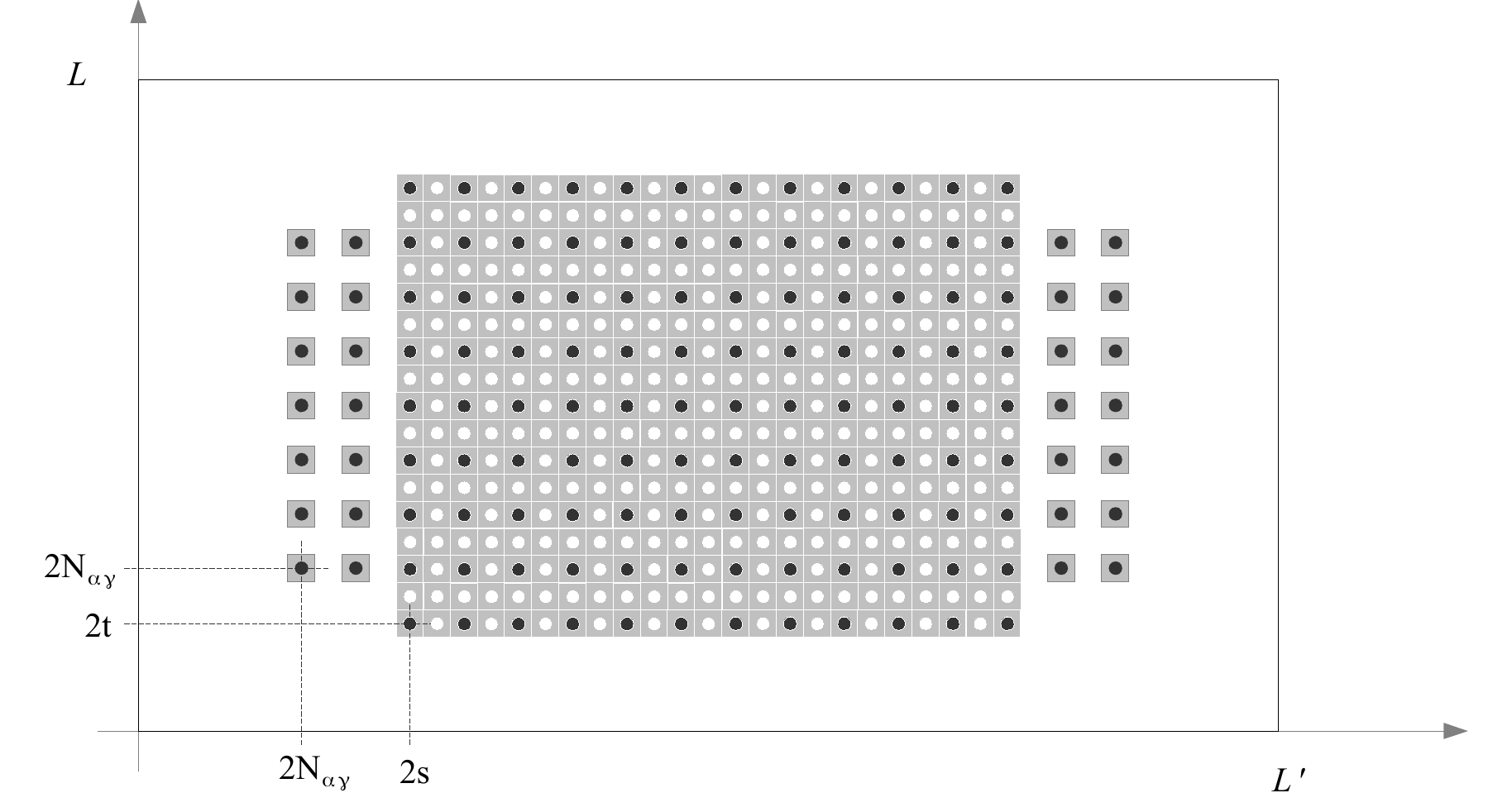}}
\caption{The set $J_\e(s,t)$ with $s\geq N_{\alpha\gamma}\geq t$}\label{Jepsilon3}
   \end{figure}
   
Hence, we get for the energy in 
$([N_{\alpha\gamma},\frac{L^\prime}{4\e}]\times [0,N_{\alpha\gamma}])\cap \mathbb N^2$ 
the following expression 
\begin{eqnarray}
g_\e^{\alpha\gamma}(s,t)&=&P_L(s-N_{\alpha\gamma})
+\frac{6L N_{\alpha\gamma}}{\gamma}(s-N_{\alpha\gamma})
+Q_L(N_{\alpha\gamma})+Q_{L^\prime}(t)\nonumber \\
&&+\frac{\e}{\gamma} r_{\alpha\gamma}-\frac{\e}{\gamma}\varrho_\e p(t)\nonumber 
\end{eqnarray}
where $P_L$ is defined by (\ref{def-PL}) as in the case $4\alpha\gamma<1$, and 
\begin{eqnarray*}
r_{\alpha\gamma}(h,t)=r(N_{\alpha\gamma},t)+8h  (N_{\alpha\gamma}-t)(2N_{\alpha\gamma}+2t+1)
+R_{\alpha\gamma}(h,0).  
\end{eqnarray*}
Note that $r_{\alpha\gamma}(h,t)$ decreases with respect to each variable in $\{h,t\geq 0\}$. 

\smallskip
The following lemma holds. 
\begin{lemma} \label{medio}
If $L\leq \lambda_c^\ast$, the unique minimum point of $g^{\alpha\gamma}_\e(s,t)$ 
in the set 
$([N_{\alpha\gamma},\frac{L^\prime}{4\e}]\times [0,N_{\alpha\gamma}])\cap \mathbb N^2$ 
 is $(\big\lfloor\frac{2\beta\gamma}{3L}-\frac{2\alpha\gamma}{3}+\frac{1}{6}\big\rfloor,\varphi(L^\prime)\wedge N_{\alpha\gamma})$.
\end{lemma}

The proof follows as in the previous cases by uniform convergence; the minimum 
is obtained in the set of the points in 
$([N_{\alpha\gamma},\frac{L^\prime}{4\e}]\times [0,N_{\alpha\gamma}])\cap \mathbb N^2$
minimizing the distance 
from 
$(m(L),\mu(L^\prime))$. 

Again, the monotonicity of $r_{\alpha\gamma}$ 
implies that the result holds 
also when $m(L)+\frac{1}{2}$ or 
$\mu(L^\prime)+\frac{1}{2}$ 
are integer. 

\bigskip 

If $(s,t)\in ([0,N_{\alpha\gamma}]\times[N_{\alpha\gamma},\frac{L}{4\e}])
\cap \mathbb N^2$, 
following the argument of the previous case the expression for $g_\e^{\alpha\gamma}$ turns out to be 
\begin{eqnarray}
g_\e^{\alpha\gamma}(s,t)&=&Q_L(s)+P_{L^\prime}(t-N_{\alpha\gamma})
+\frac{6L^\prime N_{\alpha\gamma}}{\gamma}(t-N_{\alpha\gamma})
+Q_{L^\prime}(N_{\alpha\gamma})\nonumber \\
&&+\frac{\e}{\gamma} r_{\alpha\gamma}(t-N_{\alpha\gamma},s)-\frac{\e}{\gamma}\varrho_\e \pi_{\alpha\gamma}(t)\nonumber 
\end{eqnarray}
where 
$\pi_{\alpha\gamma}(k)=
 3k^2 
+2(3N_{\alpha\gamma}+2\alpha\gamma+1)k
+p(N_{\alpha\gamma}).$ 

\medskip 

The following lemma holds.   
\begin{lemma}\label{medio2}
If $L^\prime\leq \lambda_c^\ast$, the unique minimum point of $g^{\alpha\gamma}_\e(s,t)$ 
in $([0,N_{\alpha\gamma}] \times [N_{\alpha\gamma},\frac{L}{4\e}])\cap \mathbb N^2$ 
is $(N_{\alpha\gamma}, 
\big\lfloor\frac{2\beta\gamma}{3L^\prime}-\frac{2\alpha\gamma}{3}+\frac{1}{6}\big\rfloor)=
(\varphi(L)\wedge N_{\alpha\gamma}, 
\big\lfloor\frac{2\beta\gamma}{3L^\prime}-\frac{2\alpha\gamma}{3}+\frac{1}{6}\big\rfloor)$. 
\end{lemma}

\begin{remark}\label{dis-1}
{\rm If $L>\lambda_c^\ast$, 
then the minimum 
of $g^{\alpha\gamma}_\e(s,t)$ in 
$([N_{\alpha\gamma},\frac{L^\prime}{4\e}]\times [0,N_{\alpha\gamma}])\cap \mathbb N^2$ 
is achieved 
for $s=N_{\alpha\gamma}$. This implies that the minimum point of the energy in 
$([0,\frac{L^\prime}{4\e}]\times [0,N_{\alpha\gamma}])\cap \mathbb N^2$ 
is unique and belongs to 
$[0,N_{\alpha\gamma}]^2\cap \mathbb N^2.$ 
The corresponding result holds if $L^\prime>\lambda_c^\ast$, so that in this case 
the minimum point in 
$([0,N_{\alpha\gamma}] \times [0,\frac{L}{4\e}])\cap \mathbb N^2$ 
is unique and belongs to 
$[0,N_{\alpha\gamma}]^2\cap \mathbb N^2$.
}
\end{remark} 

\medskip 

The thesis of Theorem \ref{teo-n} follows by comparing the results of Lemmas \ref{minore}, \ref{medio}, \ref{medio2}, and \ref{maggiori}, and 
recalling Remark \ref{dis-1}. 
\end{proof}

We can use the set $I_\e(n(L),n(L^\prime))$ as a recursive datum for the energy (\ref{energia}).
Note that the weak sites of the initial configurations not in $R_n$ will disappear at the following step 
if $4\alpha\gamma>1$, while the case  $4\alpha\gamma=1$ is exceptional and we may keep or discard any of such weak sites in the following. In any case we can describe the motion through the velocity of the moving rectangle corresponding to $C_\e(h,k)$.

%

%

We finally compute the velocity of a side.
We consider an edge with initial non-scaled length $L$. Theorem \ref{teo-n} shows that it moves only 
if $L<\lambda^+$ (also for $L=\lambda^+$ if $L$ is not the minimal length of the edges, or if $4\alpha\gamma>2$). In this case, the displacement of the edge is given by $
2\varphi(L)$
Hence, the velocity of each side is given by
$$v_\gamma(L)=-\frac{2}{\gamma}
\varphi(L).
$$ 

\begin{remark}[the limit case $\gamma\to+\infty$]\rm
Since $\lim\limits_{\gamma\to\infty}\frac{2N_{\alpha\gamma}}{\gamma}= 4\alpha$, we have 
$$\lambda^\ast_c=\lambda^\ast_c(\gamma)=\frac{4\beta\gamma}{4\alpha\gamma+6N_{\alpha\gamma}+5}\to 
\frac{\beta}{4\alpha}.$$
Choosing 
$$L< \frac{\beta}{4\alpha}$$ 
we get the limit for $\gamma\to+\infty$  
$$\lim_{\gamma\to+\infty}v_L(\gamma)=-\frac{4\beta}{3L}+\frac{4\alpha}{3}.$$
\end{remark}

\subsection*{Acknowledgments}
The first author acknowledges the hospitality of the Mathematical Institute in Oxford, where a substantial part of this work has been carried out. The authors are grateful to R.D.~James for drawing their attention to the connections of this work with problems in Fluid Mechanics.

\end{document}